\documentclass{amsart}

\usepackage{amsfonts,amssymb,amsmath,amsthm}
 \usepackage[utf8]{inputenc}
\usepackage{a4}
\usepackage[francais,english]{babel}
\theoremstyle{defiCnition}
\newtheorem{theo}{Theorem}
\newtheorem{rema}{Remark}
\newtheorem{coro}{Corollary}
\newtheorem{lemma}{Lemma}

\def\suny{\sum_{n=1}^\infty}
\def\suzy{\sum_{k=0}^\infty}
\def\suky{\sum_{k=m}^\infty}

\def\Nint{\mathbb{N}^*}

\def\Cplx{\mathbb{C}}
\def\RE#1{\Re{e}(#1)}
\newcommand{\biindice}[3]%
{

\begin{array}[t]{c}
#1\\
{\scryptstyle #2}\\
{\scryptstyle #3}
\end{array}

}

\def\newblock {\hskip .11em plus .33em minus .07em}

\author{Claude Henri Picard}
\address{58 avenue de Mortières\\ 71640 Givry, France}
\email{claudeh5@free.fr}
\date{28 may 2015}
\title{On some series formed by values of the Riemann Zeta function}

\begin{document}

\begin{abstract}
The partial fraction expansion of $\coth(\pi z)$, due to Euler, is generalized
to power series having for coefficients the Riemann zeta function evaluated
at certain arithmetic sequences.
A further generalization using arbitrary Dirichlet series is also proposed.
The resulting formulas are new, as far as we know, since they could not be found in any of the classical
or recent handbooks of formulas that were at our disposal.
\end{abstract}

\maketitle

\section{Origin of the series}
The starting point for this paper is the known formula \cite[p. 49]{Valiron:1942}:
$$
	\suny \frac1{n^2+z^2} = \begin{cases}
		\displaystyle \frac{\pi^2}{6} & \text{if }z=0,
		\\
		\displaystyle \frac{\pi}{2z} \coth (\pi z) - \frac1{2z^2} & \text{if }0<|z|<1.
					\end{cases}
$$
These two values must obviously be connected. Indeed, from the Taylor series for $z=0$, we obtain the expansion
$$
	\suny \frac1{n^2+z^2} = \frac{\pi^2}{6} - \frac{\pi^4}{90}z^2 + \frac{\pi^6}{945}z^4 + O(z^6),
		\qquad(z\to0),
$$
where the classical values of the Riemann's zeta function for the even integers can be
recognized immediately. The same question can now be asked for similar partial fraction expansions such as
$$
	\suny \frac1{n^3+z^3}, \quad\suny \frac1{n^4+z^4}, \quad\ldots
$$

Our results presented below extend significantly some particular formulas published recently in \cite{ChoiSrivastava:1997} and \cite{Jameson:2014} (although it was written before knowing these papers), notably by considering
Cesàro summation and more general Dirichlet series.

\section{Formula of order $p$}
\begin{theo} 
For a complex $z$ of modulus less or equal to $1$, $z \neq -1$:
\begin{equation} \label{n2plusz}
	\suny {\frac1{n^2+z}} = \suzy {(-1)^k \zeta(2k+2) z^{k}},
\end{equation}
\begin{equation} \label{n3plusz}
	\suny {\frac1{n^3+z}} = \suzy {(-1)^k \zeta(3k+3) z^{k}},
\end{equation}

and generally for $p\in\Cplx$ such that $\RE{p}>1$ ({\it formula of order $p$})~:

For all complex numbers $z$ of modulus less or equal to $1$, $z \neq -1$:

\begin{equation} \label{npplusz}
	\suny {\frac1{n^p+z}} = \suzy {(-1)^k \zeta(pk+p) z^{k}}.
\end{equation}
\end{theo}
\begin{proof} We prove only the formula of order $p$, since the first two formulas are  special cases.
Let us first assume $|z|<r_0<1$. This condition implies the normal convergence\footnote{   \cite[p. 124]{Bromwich:1926}
``Series wich satisfy the M-test have been called normally convergent by Baire."} of the series and hence the legitimate inversion of the summations
\cite[p. 104]{Remmert:1998}.

We have
$$\suny \frac1{n^p+z^p} =  \suny \frac1{n^p}\frac1{1+(z/n)^p} = \suny \frac1{n^p} \suzy (-1)^k \left(\frac{z}{n}\right)^{kp}$$ 
and we obtain, with the inversion of summations,
$$\suny \frac1{n^p+z^p} =\suzy (-1)^k z^{pk} \suny \frac1{n^{p(k+1)}}=\suzy {(-1)^k \zeta(pk+p) z^{pk}}$$
Substituting $z$ to $z^p$ yields the formula of order $p$.
We thus have an analytical function of $z$ in $D(0,r_0)$ which can be extended to the whole open disc $D(0,1)$ because the convergence radius is $1$ from the Cauchy-Hadamard formula.
\end{proof}

On the convergence circle, the analytical function 
$$
	f_p(z)=\suzy {(-1)^k \zeta(pk+p) z^{pk}}
$$
has a pole at $z=-1$. From the formula, the only problematic term in the left hand side is the first term whose denominator is zero if $z=-1$.
We conclude that the function $f_p(z)$ has a simple pole on the convergence circle 
$|z|=1$ at $z=-1$.

\begin{theo} The function $f_p(z)$ is meromorphic on $\mathbb{C}$, and its singularities are an infinity of simple poles at $z=-n^p$ ($n \in \Nint$) with residue $1$.
\end{theo}
\begin{proof} All the singularities of $f_p$ must appear in the left hand side of the formula of order $p$. We begin by the pole at $-1$.
We have to calculate the residue by the formula 
$$
	\lim_{z \rightarrow -1} (z+1)f_p(z).
$$

From the formula of order $p$, we have 
\begin{align*}
\lim_{z \to -1} (z+1)f_p(z) &= \lim_{z \to -1}\suny {\frac1{n^p+z}}
\\
	&=\lim_{z \to -1} (z+1)\frac1{1+z}
		+ \lim_{z \to -1} (z+1)\sum_{n=2}^\infty {\frac1{n^p+z}}
\\
	&= 1 + 0 = 1.
\end{align*}
Hence this pole is simple and has residue $1$.
For the other poles, the same demonstration can be given by considering $n^p+z$ instead of $1+z$.
\end{proof}

\begin{rema} When $p$ is not real, the poles of $f_p(z)$ are on a logarithmic spiral.
The poles are on a straight line (the real axis) only when $p$ is real.
It is the only analytical function that we know of whose poles are on a spiral,
without being specially constructed by using the Mittag-Leffler's theorem on the principal parts.
\end{rema}

\section{First generalization}

Let $m$ and $p$ be two complex numbers such that $\RE{m} < \RE{p}-1$. We have, for $|z|<1$,
$$
	\suny \frac{n^m}{n^p+z^p} =  \suny \frac1{n^{p-m}}\frac1{1+(z/n)^p} 
	= \suny \frac1{n^{p-m}} \suzy (-1)^k \left(\frac{z}{n}\right)^{kp}
$$
and then, since the condition on $m$ ensures the absolute convergence of the series,
$$
	\suny \frac{n^m}{n^p+z^p} = \suzy (-1)^k z^{pk} \suny \frac1{n^{p(k+1)-m}},
$$
Hence
\begin{theo} Let $m$ and $p$ be two complex numbers such that $\RE{m} < \RE{p}-1$. 
Then, for $|z| \leq 1$, $z \neq -1$,
\begin{equation}
	\suny \frac{n^m}{n^p+z} =  \suzy {(-1)^k \zeta(pk+p-m) z^{k}},
\end{equation}
and the power series 
$$
	\suzy {(-1)^k \zeta(pk+p-m) z^{k}}
$$
defines a meromorphic function which can be extended to the whole complex plane, 
except for an infinity of simple poles at $z=-n^p$ ($n \in \Nint$) with residue $n^m$.
\end{theo}

\section{Second generalization}

Let $m$ be an integer. For $|z| < 1$, we have
$$
	\suny \frac{\ln^m(n)}{n^p+z^p} =  \suny \frac{\ln^m(n)}{n^{p}}\frac1{1+(z/n)^p}
	 = \suny \frac{\ln^m(n)}{n^{p}} \suzy (-1)^k \left(\frac{z}{n}\right)^{kp}
$$
and then
$$
	\suny \frac{\ln^m(n)}{n^p+z^p} =
	  \suzy (-1)^k z^{pk} \suny \frac{\ln^m(n)}{n^{p(k+1)}}.
$$
Hence
\begin{theo} Let $m$ an integer. For $|z| \leq 1$, $z \neq -1$,
\begin{equation}
	\suny \frac{\ln^m(n)}{n^p+z} =  (-1)^m \suzy {(-1)^k \zeta^{(m)}(pk+p) z^{k}},
\end{equation}
and the power series 
$$
	\suzy {(-1)^k \zeta^{(m)}(pk+p) z^{k}}
$$
defines a meromorphic function which can be extended to the whole complex plane,
except for an infinity of simple poles at $z=-n^p$ ($n \in \Nint$) with residue $(-1)^m \ln^m(n)$.
\end{theo}

By combining the two previous theorems, we get a more general result.
\begin{theo} Let $m$ be an integer, and $p$, $q$ be two complex numbers such that $\RE{q} < \RE{p}-1$.

For $|z| \leq 1$, $z \neq -1$,
\begin{equation}
	\suny \frac{n^q\ln^m(n)}{n^p+z} =  (-1)^m \suzy {(-1)^k \zeta^{(m)}(pk+p-q) z^{k}},
\end{equation}
and the power series 
$$
	\suzy {(-1)^k \zeta^{(m)}(pk+p-q) z^{k}}
$$
defines a meromorphic function which can be extended to the whole complex plane, 
except for an infinity of simple poles at $z=-n^p$ ($n \in \Nint$) with residue $(-1)^m n^q\ln^m(n)$.
\end{theo}
\section{Third generalization}
The same principle works with more complex expressions.

We have, with two terms, for $|z|<1/\max(|a|,|b|)$, $\RE{p}>1$, $\RE{q}>1$, 
\begin{equation}
	\suny \frac{1}{(n^p+az)(n^q+bz)} =  \sum_{l=0}^\infty \suzy {(-1)^{k+l} \zeta(p+q +pk+lq) a^k b^l z^{k+l}}.
\end{equation}
With three terms, for $|z|<1/\max(|a|,|b|,|c|)$, $\RE{p}>1$, $\RE{q}>1$, $\RE{r}>1$,
\begin{multline}
	\suny \frac{1}{(n^p+az)(n^q+bz)(n^r+cz)} =\\
	\sum_{m=0}^\infty \sum_{l=0}^\infty \suzy {(-1)^{k+l+m} \zeta(p+q+r +pk+lq+mr) a^k b^l c^m z^{k+l+m}}.
\end{multline}
And so on.

\section{General formulas}

As previously we work with expressions of the form $\sum a_n/(b_n+z)$ with $\sum a_n/b_n$ absolutely summable.
We remark that these expressions, as function of $z$, are absolutely summable for $z \in \Cplx - \{b_n | n \in \Nint\}$.

We have also
\begin{lemma} (Abel, 1826) \label{Abellem}

Let the entire series $$f(z)=\suzy a_k z^k$$ of convergence radius $1$.
If $$\suzy a_k$$ converge to $L$ then, for $z$ in an Stolz angle of vertex $1$, $$\lim_{z \to 1} \suzy a_k z^k = L=f(1).$$ 

\end{lemma}

\begin{theo} \label{thdirichlet}
Consider the Dirichlet's series
$$
	f(s)=\suny \frac{a_n}{n^s}
$$
having $\sigma_a \neq +\infty$ as abscissa of absolute convergence.
Let $s$ a complex number such that $\RE{s}>\sigma_a$.

For $|z| \leq 1$, $z \neq -1$,
\begin{equation} \label{thdirichletf}
	\suny \frac{a_n}{n^s+z} =  \suzy {(-1)^k f(ks+s) z^{k}}.
\end{equation}

The power series 
$$
	g(z)=\suzy {(-1)^k f(ks+s) z^{k}}
$$
defines a meromorphic function $g$ which can be extended to the whole complex plane, 
except for an infinity\footnote{We assume the existence of an infinite subsequence of non-zero numbers $a_n$.} of simple poles at $z=-n^s$ ($n \in \Nint$) with residue $a_n$ if $a_n \neq 0$.
\end{theo}
\begin{proof} We have, for $|z| < 1$,
$$\suny \frac{a_n}{n^s+z^s} =  \suny \frac{a_n}{n^s}\frac1{1+(z/n)^s} = \suny \frac{a_n}{n^s} \suzy (-1)^k \left(\frac{z}{n}\right)^{ks}$$ 
and we obtain, with the inversion of summations, legitimated by absolute convergence,
$$\suny \frac{a_n}{n^s+z^s} =\suzy (-1)^k z^{sk} \suny \frac{a_n}{n^{s(k+1)}}=\suzy {(-1)^k f(ks+s) z^{sk}}$$
Substituting $z$ to $z^s$ yields the formula (\ref{thdirichletf}).

To extend to $|z| \leq 1$, $z \neq -1$, we apply the Abel's lemma \ref{Abellem}.

All the singularities of $g$ must appear in the left hand side of (\ref{thdirichletf}).

We have
\begin{align*}
 \lim_{z \to -n^s} (z+n^s)g(z) &= \lim_{z \to -n^s}\suny {\frac{a_n}{n^s+z}}
 \\
 	&=\lim_{z \to -n^s} (z+n^s)\frac{a_n}{n^s+z}
 		+ \lim_{z \to -n^s} (z+n^s)\sum_{k \neq n} {\frac{a_k}{k^s+z}}
 \\
 	&= a_n + 0 = a_n.
\end{align*}
Hence $z=-n^s$ is a simple pole if $a_n \neq 0$.

For a point $\alpha$ not of the form $-n^s$, $$\lim_{z \to \alpha} (z-\alpha)g(z) = \lim_{z \to \alpha}(z-\alpha)\suny {\frac{a_n}{n^s+z}} = 0,$$ so there is not other singularity in $\Cplx$.
\end{proof}
\begin{rema}
We can work with expressions as
$$\suny \frac{a_n}{n^s(n^s+z)} =  \suzy {(-1)^k f(ks+2s) z^{k}},$$
or
$$\suny \frac{a_n}{n^{s_1}(n^{s_2}+\alpha z)} =  \suzy {(-1)^k f(ks_2+s_1+s_2) \alpha^k z^{k}},$$
and so on.
\end{rema}
\begin{coro} Under the previous hypotheses, for an integer $m \ge 0$, we have
\begin{equation}
	\suny \frac{a_n}{(n^s+z)^{m+1}} =  \suky {(-1)^{k+m} f(ks+s) \frac{k!}{m!(k-m)!} z^{k-m}}.
\end{equation}
\end{coro}

\begin{proof} This is the derivative of order $m$ of the formula (\ref{thdirichletf}) with respect to $z$.\end{proof}
\begin{coro} Let $(\alpha_i)$ a sequence of $m>1$ complex numbers, $(\beta_i)$ a sequence of $m$ complex numbers such that $\RE{\beta_i}>1$  and  $f$ as in the previous theorem \ref{thdirichlet}.

We have, for $|z| < 1/\max_{i=1, \ldots, m} (\alpha_i)$
\begin{multline}
	\suny a_n \prod_{i=1}^m \frac1{(n^{\beta_i}+\alpha_i z)} =\\
	\sum_{m_1=0}^\infty \sum_{m_2=0}^\infty \ldots \sum_{m_m=0}^\infty{(-1)^{\sum_{i=1}^m {m_i}} f\left(\sum_{i=1}^m (m_i+1) \beta_i \right) \left(\prod_{i=1}^m \alpha_i^{m_i}\right) z^{\sum_{i=1}^m m_i}}.
\end{multline}
\end{coro}
\begin{rema} If $\alpha_l=0$ then $m_l=0$ in the formula. \end{rema}
\begin{theo} \label{thserie}
Let the analytical function $f(s)$ be defined in the disc $D(0,R)$, $R>0$, by its power series~:
$$
	f(s) = \sum_{k=1}^\infty a_k s^k,
$$
and zero in $0$.

Fix a complex $s$ such that $\RE{s} > 1$.

For $|z| <R$,
\begin{equation}
	\suny f\left(\frac{z}{n^s}\right) =  \sum_{k=1}^\infty {a_k \zeta(ks) z^{k}}.
\end{equation}

If $a_1=0$, and $s$ such that $\RE{s} > \frac12$,
\begin{equation}
	\suny f\left(\frac{z}{n^s}\right) =  \sum_{k=2}^\infty {a_k \zeta(ks) z^{k}}.
\end{equation}

\end{theo}
\begin{rema} This theorem contents the formulas of \cite{Jameson:2014} when $s=1$ and $z=1$ or $z=-1$. 
\end{rema}
\begin{proof} The conditions ensure the absolute convergence of the series. We have
\begin{align*}
	\suny f\left(\frac{z}{n^s}\right) &=  \suny \sum_{k=1}^\infty a_k\left(\frac{z}{n^s}\right)^k\\
	                                  &=  \sum_{k=1}^\infty {a_k \suny \left(\frac{z}{n^s}\right)^k}\\
	                                  &=  \sum_{k=1}^\infty {a_k \zeta(ks) z^{k}}.
\end{align*}
\end{proof}
By combining the last two theorems, we get a more general result.
\begin{theo} \label{thdirichletserie}
Let the analytical function $f$, defined in the disc $D(0,R)$ by the series
$$
	f(z) = \sum_{k=1}^\infty{a_k z^k},
$$
and zero in $0$.

Let $g(s)$ the complex analytical function defined by the Dirichlet's series
$$
	g(s)=\suny \frac{b_n}{n^s}
$$
in the half-plane $\RE{s}>\sigma_a$.

Let a complex number $s$ such that $\RE{s} \ge 1$, $|z|<R$, another number $s'$ such that $\RE{s+s'}>\sigma_a$.

We have
\begin{equation}
	\suny \frac{b_n}{n^{s'}}f\left(\frac{z}{n^s}\right) =  \sum_{k=1}^\infty {a_k  g(ks+s')} z^{k}.
\end{equation}
\end{theo}
\begin{theo} \label{thserie2}
Let the analytical function $f(s)$ be defined in the disc $D(0,R)$, $R>0$, by its power series~:
$$
	f(s) = \sum_{k=1}^\infty a_k s^k,
$$
and zero in $0$.

Let $(\lambda_n)$ a sequence of positive real numbers whose are increasing up to infinity 
 such that the sum $$\suny |\exp(-\lambda_n s)|$$ converges for $\RE{s}>\sigma_a$.
 
Let $$D(s)=\suny \exp(-\lambda_n s)$$ and a complex $s$ such that $\RE{s} > \sigma_a$.

For $|z| <R$,
\begin{equation}
	\suny f\left(ze^{-\lambda_n s}\right) =  \sum_{k=1}^\infty {a_k D(ks) z^{k}}.
\end{equation}
\end{theo}
\begin{theo} \label{thsuite} Let $(b_n)$ a sequence of complex numbers whose moduli are increasing up to infinity 
and $(a_n)$ an infinite sequence of complex numbers such that the sum $\suny |{a_n}/{b_n}|$ converges.

Then 
\begin{equation} \label{thsuitef}
	\suny{\frac{a_n}{b_n-z}} = \suzy{\left(\suny \frac{a_n}{b_n^{k+1}}\right) z^k}
\end{equation}
for $|z| \leq |b_1|$, $z \neq b_1$.

The power series 
$$
	\suzy{\left(\suny \frac{a_n}{b_n^{k+1}}\right) z^k}
$$
defines a meromorphic function which can be extended to the whole complex plane,
except for simple poles at $z=b_n$ ($n \in \Nint$) with residue $-a_n$ when $a_n \neq 0$.
\end{theo}
\begin{coro} Under the previous hypothesis, for an integer $m \ge 0$, we have
\begin{equation} 
        \suny{\frac{a_n}{(b_n-z)^{m+1}}} = \suky{\left(\suny \frac{a_n}{b_n^{k+1}}\right)\frac{k!}{m!(k-m)!} z^{k-m}}
\end{equation}
\end{coro}
\begin{proof} This is the derivative of order $m$ of the formula (\ref{thsuitef}) with respect to $z$.\end{proof}

\section{The inverse problem}

\section{Applications of the previous formulas}

The notation $(C,k)$ means the Cesàro summation method of order $k$ \cite[p. 96]{Hardy:1949}.

\subsection{} First formula (\ref{n2plusz}), $z=1$ in the Cesàro's sense \cite[p. 205]{Apostol:1974}
$$
	\suny \frac1{n^2+1}=\sum_{k=1}^\infty(-1)^{k+1}\zeta(2k)\qquad(C,1).
$$
\subsection{} Second formula (\ref{n3plusz}), $z=1$ in the Cesàro's sense
$$
	\suny \frac1{n^3+1}=\sum_{k=1}^\infty(-1)^{k+1}\zeta(3k)\qquad(C,1).
$$
\subsection{} Partial derivative of the first formula (\ref{n2plusz}) with respect to $z$,
$z=1$ in the Cesàro's sense
$$
	\suny \frac1{(n^2+1)^2}=\sum_{k=1}^\infty(-1)^{k+1}k\zeta(2k)\qquad(C,2).
$$
\subsection{} Theorem \ref{thserie} for $f(s)=\ln(1+s)$, $\RE{s}>1$ then limit $z \to 1^-$
$$
	\suny \ln\left(1+\frac1{n^s}\right)=\sum_{k=1}^\infty(-1)^{k}\frac{\zeta(ks)}{k}.
$$
\subsection{} Theorem \ref{thserie} for $f(s)=\exp(s)-1$
$$
	\suny \left[\exp\left(\frac1{n^2}\right)-1\right]=\sum_{k=1}^\infty \frac{\zeta(2k)}{k!}.
$$
\subsection{} Theorem \ref{thserie} for $f(z)=\sin(z)$, $\RE{s}>1$
$$
	\suny \sin\left(\frac{z}{n^s}\right) = \suzy (-1)^k \frac{z^{2k+1}}{(2k+1)!}\zeta\left(2ks+s\right).
$$
\subsection{} Theorem \ref{thdirichlet} for $f(s)=1/\zeta(s)$.
$$
        \suny \frac{\mu(n)}{n^2+z} =  \suzy { \frac{(-1)^k}{\zeta(2k+2)} z^{k}}.
$$
\subsection{} Theorem \ref{thdirichlet} for $f(s)=\zeta'(s)/\zeta(s)$.
$$
        \suny \frac{\Lambda(n)}{n^2+z} =  \suzy {(-1)^{k+1} \frac{\zeta'}{\zeta}(2k+2) z^{k}}.
$$
where $$\Lambda(n)=\begin{cases}
		\displaystyle \ln(p) & \text{if } n=p^m,\quad  p\text{ prime,}
		\\
		\displaystyle 0 & \text{otherwise.}
					\end{cases}$$
\subsection{} Theorem \ref{thdirichlet} for $f(s)=\zeta(s-1)/\zeta(s)$.
$$
        \suny \frac{\varphi(n)}{n^3+z} =  \suzy {(-1)^{k+1} \frac{\zeta(3k+2)}{\zeta(3k+3)} z^{k}}.
$$
where $\varphi(n)$ is the Euler's totient function.

So
$$
        \suny \frac{\varphi(n)}{n^3+1} =  \suzy {(-1)^{k+1} \frac{\zeta(3k+2)}{\zeta(3k+3)}}.
$$

\subsection{} We can use the theorems \ref{thdirichlet}, \ref{thdirichletserie} and the two corollaries with the L-Dirichlet functions.
With $a_n=\chi(n)$, we have by the theorem \ref{thdirichlet}
$$\sum_{n=1}^\infty \frac{\chi(n)}{n^s+z} = \sum_{k=1}^\infty (-1)^k L(ks+s,\chi)z^k.$$
and by the theorem \ref{thdirichletserie}
$$\suny \frac{\chi(n)}{n^{s'}}f\left(\frac{z}{n^s}\right) =  \sum_{k=1}^\infty {a_k  L(ks+s',\chi)} z^{k}.$$
For example, with the Catalan's zeta-function
$$\beta( s)= \sum_{n=0}^\infty \frac{(-1)^n}{(2n+1)^s},$$
we obtain $$\sum_{n=0}^\infty \frac{(-1)^n}{(2n+1)^s+z} = \sum_{k=1}^\infty (-1)^k\beta(ks+s)z^k.$$

So
$$\sum_{n=1}^\infty \frac{(-1)^n}{(2n+1)^s+z} = \sum_{k=1}^\infty (-1)^k \left[\beta(ks+s)-1\right]z^k.$$
\subsection{} From 
$$
	\suny \frac1{4n^2-1} = \frac12,
$$
and $z=-1/4$ in the first formula (\ref{n2plusz}), we have
$$
	\suny \frac{\zeta(2n)}{4^n} = \frac12.
$$

\subsection{} With $z=-1/16$, the first formula (\ref{n2plusz}) gives
$$
	\suny \frac{\zeta(2n)}{16^n} = \frac1{2} - \frac{\pi}{8}.
$$

\subsection{} Some similar formulas are already known, for example 
\cite[p. 691, \S5.1.27]{BrychkovMarichevPrudnikov:1998} (with the
sign of the fraction $\frac12$ corrected to be negative)
$$
	\suny \frac1{n^4+1}  = \suzy (-1)^k{\zeta(4k+4)}
					= -\frac12+\frac{\pi\sqrt{2}}{4}\frac{\sinh(\pi\sqrt{2})
						+\sin(\pi\sqrt{2})}{\cosh(\pi\sqrt{2})-\cos(\pi\sqrt{2})}.
$$
\subsection{} Theorem \ref{thsuite} for $a_n=1$, $b_n=n^2+n$, $z=1$

$$\suny{\frac{1}{n^2+n-1}} =1+\frac{\sqrt{5}}{5}\pi \tan \left(\frac{\pi \sqrt{5}}{2}\right)= \suzy{\left(\suny \frac{1}{(n^2+n)^{k+1}}\right)}$$

\section{Code for Maple}

For theorem \ref{thserie}
\begin{verbatim}
#Definition of the function
f := x -> exp(x);

#The function must be 0 for x=0
g:= x -> f(x ) - f( 0);

#We calculate the coefficient of the entire series
a := n->eval(diff(g(x), x$n), x = 0)/factorial(n);
 
#Values of s and z
z := 1/3; s := 2;
 
#Left hand side
s1 := sum(g(z/n^s), n = 1 .. infinity);            evalf(s1, 20);
 
#Right hand side
ss2 := sum(a(k)*Zeta(k*s)*z^k, k = 1 .. infinity); evalf(ss2, 20);
\end{verbatim}
 
\bigskip\noindent Acknowledgements.
The author would like to thank Jacques Gélinas for his help with this English translation
of the original French version of the paper and the bibliography.

\vfill
\end{document}